\documentclass{article} 
\usepackage{amsmath,amsthm,amssymb}  
\usepackage{mathdots}

\theoremstyle{plain}
\newtheorem{theorem}{Theorem}[section]
\newtheorem{proposition}[theorem]{Proposition}         
\newtheorem{corollary}[theorem]{Corollary} 
\newtheorem{lemma}[theorem]{Lemma} 
\newtheorem{definition}[theorem]{Definition}   

\theoremstyle{remark}  
\newtheorem{example}[theorem]{Example}

\newcommand{\C}{\mathbb C}   

\newcommand{\Z}{\mathbb Z}

\renewcommand{\P}{\mathbb P}

\newcommand{\al}{\alpha}
\newcommand{\be}{\beta} 
\newcommand{\ga}{\gamma}
\newcommand{\de}{\delta}

\newcommand{\La}{\Lambda}
\newcommand{\Om}{\Omega}
\newcommand{\De}{\Delta}
\newcommand{\om}{\omega}

\DeclareMathOperator{\diag}{diag}
\DeclareMathOperator{\age}{age}
\DeclareMathOperator{\orbi}{orbi}

\newcommand{\calM}{\mathcal{M}}
\newcommand{\calF}{\mathcal{F}}

\newcommand{\sub}{\subseteq}  
\newcommand{\st}{\ \vert\ }   
\renewcommand{\ll}{\lq\lq}
\newcommand{\rr}{\rq\rq\ }
\newcommand{\rrr}{\rq\rq} 
\newcommand{\lan}{\langle}
\newcommand{\ran}{\rangle}  
\renewcommand{\b}{\partial}
\newcommand{\lann}{\langle\!\langle}
\newcommand{\rann}{\rangle\!\rangle}

\newcommand{\HO}{H_{\orbi}}
\newcommand{\PW}{\P(w)}
\newcommand{\one}{{\mathbf 1}}

\newcommand{\s}{{s}} 
\newcommand{\h}{ {\hbar} }

\begin{document}     


\title{Orbifold quantum D-modules associated to
weighted projective spaces}  

\author{Martin A. Guest and Hironori Sakai}      

\date{}

\maketitle 

\begin{abstract}
We construct in an abstract fashion (without using Gromov-Witten invariants)
the orbifold quantum cohomology of weighted projective space, starting from a certain differential operator.  We obtain the product, grading, and intersection form by making use of the associated self-adjoint D-module and the Birkhoff factorization procedure.  The method extends in principle to the more difficult case of Fano hypersurfaces in weighted projective space, where Gromov-Witten invariants have not yet been computed, and we illustrate this by means of an example originally studied by A.~Corti. In contrast to the case of weighted projective space itself or the case of a  Fano hypersurface in projective space, a \ll small cell\rr of the Birkhoff decomposition plays a role in the calculation.
\end{abstract}

\section{Introduction}\label{intro}

The  weighted projective space
\[
\P(w_0,\dots,w_n)=
\C^{n+1}-\{0\}\ /\ \C^\ast,
\ \ 
z\cdot(z_0,\dots,z_n)=(z^{-w_0}z_0,\dots,z^{-w_n}z_n)
\]
provides a simple test case 
(see \cite{CCLTXX}, \cite{BoMaPeXX}, \cite{MaXX})  
for the recently developed theories of orbifold cohomology and orbifold quantum cohomology.
Direct geometrical calculations are difficult, but mirror symmetry suggests an alternative and very effective approach:  Corti and Golyshev conjectured (see \cite{GoXX}, \cite{CoGoXX}) that the structure constants can be read off from 
\[
T_w-q=\prod_{i=0}^n (w_i \h\b) (w_i \h\b - \h)\dots
(w_i \h\b - (w_i-1)\h)\ -\ q,
\]
where $\b =q\frac{d}{d q}$; this is an ordinary differential operator of order $\s=\sum _{i=0}^n w_i$.

This generalizes the well known quantum differential equation of projective space $\C P^n=\P(1,\dots,1)$.   Namely, the equation
$((\h \b)^{n+1} - q)y=0$ is a scalar form of the system
\[
\h \b
\begin{pmatrix}
y_0 \\  \vdots\\ \vdots \\  y_n
\end{pmatrix}
=
\begin{pmatrix}
0 & & & q\\
1 & \ddots & & \\
  & \ddots  & \ddots & \\
  & & 1 & 0
\end{pmatrix}
\begin{pmatrix}
y_0 \\  \vdots \\ \vdots\\  y_n
\end{pmatrix}
\]
where the matrix is interpreted as that of quantum multiplication by the generator $p\in H^2\C P^n$ with respect to the standard cohomology basis $1,p,\dots,p^n$.   Thus,
\[
p\circ p^i =
\begin{cases}
p^{i+1}\quad\textrm{if}\ 0\le i < n\\
q\quad\textrm{if}\ i=n
\end{cases}
\]
from which all quantum products $p^i\circ p^j$ can be computed.

The conjecture of Corti and Golyshev was proved in
\cite{CCLTXX}, by extending to orbifold quantum cohomology a method of Givental for quantum cohomology.    The method has three steps.  First,  a basis of solutions of the quantum differential equation is written down --- the $I$-function.   Then,  the orbifold version of Givental's Mirror Theorem shows that the $I$-function is equal to the $J$-function, a certain generating function for Gromov-Witten invariants. 
This is the most substantial ingredient, but specific properties of weighted projective spaces are not required.
Finally, the structure constants for the orbifold quantum product are extracted from this $J$-function by a method which involves repeated differentiation.

The first goal of this paper is to give a straightforward version (alluded to in the introduction to \cite{CCLTXX}) of the last step, using the Birkhoff factorization method of 
\cite{Gu05}.  The simplifying feature is that we use the differential equation (D-module) directly,  rather than its solution ($I$-function).

The second goal is to study in its own right the differential operator $T_w-q$, or rather, the D-module $D^\h/(T_w-q)$ (where $D^\h$ is a certain ring of differential operators).  We show how to extract from this D-module an \ll abstract quantum cohomology ring\rr with a product operation, grading, and nondegenerate 
pairing. Then we observe that this coincides with the usual orbifold quantum cohomology.  It is remarkable that such a simple differential operator contains all relevant geometrical information, which is complicated and non-intuitive even in the case of $\P(w_0,\dots,w_n)$.  

The third and main goal (section \ref{hyper})
is to indicate how our method extends to hypersurfaces in weighted projective spaces.
This generalizes the method of \cite{Sa08} for hypersurfaces in projective spaces.  It presents a new feature:  instead of the \ll big cell\rr of the Birkhoff decomposition, in general a \ll small cell\rr is needed. Alternatively, this method can be interpreted as the Gram-Schmidt orthogonalization procedure together with a \ll big cell factorization\rrr. As a nontrivial example, we apply the method to a hypersurface of degree $3$ in $\P(1,1,1,2)$, where the orbifold quantum cohomology has been computed geometrically by Corti.  However, we are not able to give general conditions which ensure that our method works, and we must leave this as a problem for the interested reader. 

The first author is very grateful to Alessio Corti for explaining the conjecture and the basic ideas of orbifold quantum cohomology; the idea for extracting the structure constants of $\P(w_0,\dots,w_n)$ from the differential operator was originally worked out with him in 2006, 
and Alessio also explained the geometry behind the hypersurface example in section \ref{hyper}.  He also thanks Hiroshi Iritani for many essential explanations and comments on an earlier version, and Josef Dorfmeister for  discussions on the Birkhoff decomposition.

The authors apologise for the long delay in preparing the final version of this article since its submission to the arXiv in 2008.
Detailed comments and suggestions by the referee are gratefully acknowledged. 

\section{Notation for orbifold cohomology}\label{notation}

We write
$
\P(w_0,\dots,w_n)=\PW
$
from now on.  As far as possible we shall follow the notation of \cite{CCLTXX} for orbifold cohomology. That paper and its references contain more detailed information.

First, let
\begin{align*}
F&=\{  \tfrac{i}{w_j} \st  0\le i \le w_j-1, \ \ 
0\le j\le n \}\\
&=\{  
f_1,\dots,f_k
\}
\text{ where $0=f_1<f_2<\dots<f_k<f_{k+1}\overset{\text{def}}{=}1$ }.
\end{align*}
Let $u_1,\dots,u_k$ be the \ll multiplicities\rr of the fractions $f_1,\dots,f_k$ as elements of $F$.  
We write
\[
\s=u_1+\cdots+u_k=w_0+\cdots+w_n.
\]
The positive integer $u_i$ can also be described as the cardinality of the set
\[
S_{f_i}=\{  j \st w_j f_i \in \Z \}
\sub \{0,\dots,n\}.  
\]

The orbifold cohomology of $\PW$ may be defined as a vector space by
\[
\HO^\ast\, \PW = \bigoplus_{i=1}^k H^\ast \P(V_{f_i}),
\]
where 
\[
V_{f_i}=\{ (z_0,\dots,z_n)\in \C^{n+1} \st  z_j=0 \text{ if } j \notin S_{f_i} \} \cong \C^{u_i}. 
\]
This can be equipped with a commutative associative multiplicative operation called the orbifold cup product. Using this product, we obtain a $\C$-basis
\[
\one_{f_i},\one_{f_i}\, p, \dots, \one_{f_i}\, p^{u_i-1}
\]
of the subspace $H^\ast \P(V_{f_i})$,
where $p\in H^2 \P(V_0)$ and $\one_{f_i}$
denotes the canonical generator of $H^0 \P(V_{f_i})$. When $i=1$ we have $f_1=0$,
$u_1=n+1$, and generators
$\one_0,\one_0 p, \dots, \one_0 p^{n}$;  we shall  just write $1,p,
\dots,  p^{n}$ in this case. The element $1$ is the identity element of the
orbifold cohomology ring.

There is also a natural grading,
in which 
\[
\vert \one_{f_i}\, p^j\vert =
\vert \one_{f_i}\vert + \vert p^j\vert=
2\age \one_{f_i} + 2j.
\]
Here,
$\age \one_{f_i}
=(u_1+\cdots+ u_{i-1})-f_i \s
=\lan -w_0f_i\ran+\cdots+\lan -w_nf_i\ran$ where
$\lan r \ran = r-\max\{i\in\Z\st i\le r  \}$.  
The orbifold cohomology has a nondegenerate symmetric \ll intersection pairing\rr $(\ ,\ )$, which generalizes the Poincar\'e pairing for ordinary cohomology.
 
We record the following properties for later use.

\begin{lemma}\label{fandu}  $ $

\noindent(1) $f_i+f_j=1$ if $i+j=k+2$.

\noindent (2) $u_i=u_j$  if $i+j=k+2$. 

\noindent(3)  $u_2+\cdots+ u_i=u_{k+2-i}+\cdots+u_k$ for $2\le i\le k$.
\end{lemma}

\begin{proof}  The involution $f\mapsto 1-f$ preserves $F\cup\{1\}$.  It maps
$f_1<\dots<f_{k+1}$ to $ 1-f_{k+1}<\dots<1-f_1$, so these sequences must coincide. This proves (1), then  (2) and (3) follow immediately.
\end{proof}

\section{The structure constants: 
statement of results}\label{results}

As mentioned in the introduction, a key role is played by the 
$\s$-th order differential operator
\begin{align*}
T_w -q
&=\prod_{i=0}^n (w_i \h\b) (w_i \h\b - \h)\dots
(w_i \h\b - (w_i-1)\h)\ -\ q
\\
&=w^w
\h^\s \prod_{i=0}^n  
\b(\b-\tfrac{1}{w_i})\dots(\b - \tfrac{w_i-1}{w_i})
\ -\ q,
\end{align*}
where $\s=\sum _{i=0}^n w_i$, $w^w=\prod _{i=0}^n w_i^{w_i}$,  and $\b =q\frac{d}{d q}$.

In this section we state without explanation how the structure constants of orbifold quantum cohomology ---  in particular, of orbifold cohomology itself --- may be extracted from the differential operator $T_w -q$.  A systematic explanation will be given in the next section.

Using the formula $\b q^{-1} = q^{-1}(\b -1)$, we may factorize 
the differential operator $q^{-1}T_w$ in the following way:
\[
q^{-1}T_w= 
\underbrace{
m_k
q^{-\De_k}
(\h\b)^{u_k}
}_{ \text{$k$th factor} }
 \
 \underbrace{
m_{k-1}
q^{-\De_{k-1}}
(\h\b)^{u_{k-1}} 
}_{ \text{$k\!\!-\!\!1$th factor} }
\ 
\ \dots\  \ 
 \underbrace{
m_1
q^{-\De_1}
(\h\b)^{u_1}  
}_{ \text{$1$st factor} } 
\]
where 
\[
\De_i=f_{i+1}-f_i,\quad
m_i=\prod_{j\in S_{f_i}} w_j,
\]
for $1\le i \le k$.
Thus we have $\prod_{i=1}^k m_i = w^w$ and
$\sum_{i=1}^k \De_i=1$.  We shall need the following symmetry properties later on, which follow directly from Lemma \ref{fandu}: 

\begin{lemma}\label{Deandm} $ $

\noindent(1)  $\De_i=\De_j$  if $i+j=k+1$.  

\noindent(2) $m_i=m_j$  if $i+j=k+2$.  
\end{lemma}

Let us rewrite the factorization above as
\[
q^{-1}T_w=
\tfrac{1}{r_\s} \ \h\b\ 
\tfrac{1}{r_{\s-1}} \ \h\b\ 
\ \ \dots\ \ 
\tfrac{1}{r_1} \h\b\ 
\]
where:

\begin{definition}\label{r}  For $1\le \al\le \s$, 
\[
r_\al=
\begin{cases}
\tfrac{1}{m_i} q^{\De_i}
\quad
\text{if $\al=u_1+\cdots+u_i$}
\\
\ \ 1\  
\quad\quad
\text{otherwise.} 
\end{cases}
\]
\end{definition}

\noindent The result of
\cite{CCLTXX} may be stated as follows:

\begin{theorem}\label{matrix}  Denote by $c_0,\dots,c_{\s-1}$ the additive basis
\[
1,p,\dots,p^{u_1-1};\ 
\one_{f_2},\one_{f_2}\, p, \dots, \one_{f_2}\, p^{u_2-1};
\ \ \dots\ ;\ \ 
\one_{f_k},\one_{f_k}\, p, \dots, \one_{f_k}\, p^{u_k-1}
\]
of $\HO^\ast\, \PW$.
Then the matrix of orbifold quantum multiplication by $p$ with respect to this basis
is given by 
\[
\begin{pmatrix}
0 & & & r_\s\\
r_1 & \ddots & & \\
 & \ddots & \ddots & \\
  & & r_{\s -1} & 0
  \end{pmatrix}.
\]
That is,  we have $p\circ c_i=r_{i+1}c_{i+1}$ for
$0\le i <\s-1$ and $p\circ c_{\s-1}=r_{\s}c_{0}$.  In particular, $p$ is a cyclic element of this ring.
\end{theorem}

The orbifold structure constants (giving the product structure of 
$\HO^\ast\, \PW$) are obtained by setting $q=0$ in the above matrix.  Although the matrix itself gives only the products involving $p$,  all other products can be deduced.

\section{Direct approach from the D-module}\label{direct}

The structure constants in Theorem \ref{matrix} were computed in \cite{CCLTXX} from the $I$-function (i.e.\ solution of the differential equation $(T_w -q) y=0$) and by establishing
a mirror theorem in the style of Givental. 
In this section we discuss a somewhat different procedure:  
we construct \ll abstract orbifold quantum cohomology\rr
from $T_w -q$ itself.  To prove that our abstract orbifold quantum cohomology agrees with the usual orbifold quantum cohomology, it is still necessary to appeal to the mirror theorem, so in this sense our procedure relates only to the extraction of information from the differential equation.  However, our procedure gives a very direct way of obtaining the orbifold degrees and orbifold Poincar\'e pairing as well as the structure constants.

We follow \cite{Gu05} and chapter 6 of \cite{Gu08}, although the orbifold case presents some new features.  Let us consider the $D^\h$-module 
\[
\calM=
D^\h/(T_w -q)
\]
where $D^\h$ denotes the ring of (ordinary) differential operators generated by $\h\b$, and  $(T_w -q)$ denotes the left ideal generated by $T_w -q$.  As coefficient ring we can take the ring of functions which are polynomial in $q^{\pm 1/l}$, where
\[
l=\text{lowest common multiple of $w_0,\dots,w_n$},
\]
and which are holomorphic in 
$\h$ in a neighbourhood of $\h=0$. 

The $D^\h$-module $\calM$ is free of rank $\s$ over the coefficient ring. With respect to the natural basis $1,\h\b,\dots,(\h\b)^{\s-1}$,  the matrix of the action of $\b$ is of the form
\[
\Om=\tfrac1\h\om+\theta^{(0)}+\h\theta^{(1)}+\h^2\theta^{(2)}+\cdots .
\]
More precisely, if we identify $\calM$ with a space of meromorphic sections of the trivial bundle, we may regard $1,\h\b,\dots,(\h\b)^{\s-1}$ as a local basis of sections, and
the action of $\b$ on $\calM$ defines a connection on the bundle, with local connection matrix $\Om$.

If we replace $\h\b$ by an abstract (commutative) variable $p$, then set $\h=0$, we may construct from $\calM$ a commutative ring generated by $p$ which is subject to the relation
$w^w p^\s - q$, and which has $1,p,\dots,p^{\s-1}$ as an additive 
$\C[q^{\pm1/l}]$-basis.  That is, we have the \ll abstract orbifold quantum cohomology ring\rr
\[
QA=\C[p,q^{\pm1/l}]/(w^w p^\s - q).
\]
In order to define \ll abstract orbifold Gromov-Witten invariants\rr (structure constants) we shall introduce a ring $A$, the \ll abstract orbifold cohomology ring\rrr, such that $QA$ and $A\otimes \C[q^{\pm1/l}]$ are isomorphic as 
$\C[q^{\pm1/l}]$-modules.  A choice of basis will give a specific isomorphism $\de: QA\cong
A\otimes \C[q^{\pm1/l}]$, hence a new
$A\otimes \C[q^{\pm1/l}]$-valued  product operation
\[
a\circ b=\de\left(  \de^{-1}(a)  \de^{-1}(b) \right)
\]
on $A$.
Our main task will be the construction of a basis such that the product has the expected properties of the orbifold quantum product.

For this, the fundamental step is to transform $1,\h\b,\dots,(\h\b)^{\s-1}$ to a new basis, with respect to which the connection matrix has the form
\[
\hat\Om =\tfrac1\h\hat\om
\]
where $\hat\om$ is independent of $\h$.
In the case of a Fano manifold, the transformation procedure is explained in detail in chapter 6 of \cite{Gu08}.  It involves a Birkhoff factorization $L=L_-L_+$ of a matrix-valued function $L$ such that
$\Om=L^{-1}dL$, after which one defines $\hat\Om=(L_-)^{-1}dL_-$.  
The basis $1,\h\b,\dots,(\h\b)^{\s-1}$ is transformed to
the new basis
$L_+^{-1}\cdot 1, L_+^{-1}\cdot\h\b,\dots, L_+^{-1}\cdot(\h\b)^{\s-1}$, where  $L_+^{-1}\cdot(\h\b)^{i}$ means $\sum_{j=0}^{s-1} (L_+)^{-1}_{ji} (\h\b)^{j}$.  

In general it is difficult to carry out such Birkhoff factorizations explicitly, or even to know whether they exist.  Here, however, $L_+$ can be found by the method of \cite{AmGu05}, \cite{Gu08}.  The effectiveness of our approach comes from the fact that only the factor $L_+$ is needed (not the more complicated factor $L_-$, which is equivalent to the $I$-function).  

In the case of weighted projective spaces themselves (though not for hypersurfaces), the differential operator factorization given in section \ref{results} provides a short cut for the computation of $L_+$.  Namely,  we introduce directly a new basis $P_0,\dots,P_{\s-1}$ by defining
\[
\text{
$P_0=1$ and
$P_i=\frac{1}{r_i} \h\b P_{i -1}$
}
\]
for $1\le i \le \s-1$.  Fortuitously, with respect to this basis, the matrix of $\b$ already has the form $\tfrac1\h\hat\om$, so $L_+$ may be read off by regarding the above basis as $L_+^{-1}\cdot 1, L_+^{-1}\cdot\h\b,\dots, L_+^{-1}\cdot(\h\b)^{s-1}$.  We have
$
L_+=Q_0(I + \h Q_1 + \cdots + \h^{k-2} Q_{k-2})
$
where
\[
Q_0=
\begin{pmatrix}
\frac{1}{m_0} q^{f_1} I & & & \\
 & \frac{1}{m_0m_1} q^{f_2} I & & \\
  & & \ddots & \\
   & & & \frac{1}{m_0\dots m_{k-1}} q^{f_k} I
\end{pmatrix},
\]
$m_0=1$, and where  $Q_1,\dots,Q_{k-2}$ are (easily computed) constant matrices.  

For future reference, we explain how (a modification of) the algorithm of \cite{AmGu05} and section 6.6 of \cite{Gu08} produces this answer.  First, by definition, the factor $L_+(q,\h)=Q_0(q)(I + \h Q_1(q) + \h^{2} Q_{2}(q)+\cdots)$ satisfies
the ordinary differential equation
\[
\tfrac1\h \hat\om \ (=\tfrac1\h Q_0\om Q_0^{-1}) = L_+ \Om L_+^{-1} + L_+ dL_+^{-1}.
\]
In the situation of \cite{AmGu05} and \cite{Gu08}, $L_+$ is determined uniquely by the initial condition $L_+\vert_{q=0}=I$, and there is a natural homogeneity condition on $L_+$ which reduces the computation of $L_+$ to a finite algebraic algorithm.   
The present situation is similar, but $L_+$ must be normalized in a different way.  

Let us make the Ansatz that $Q_0$ is of the above diagonal form. This is natural as 
$\diag(q^{f_1}I,\dots,q^{f_k}I)$ arises from the Frobenius method for solving the original o.d.e.,
and, as we shall see in Corollary \ref{metric} below, the coefficients $\frac{1}{m_0},\dots,\frac{1}{m_0\dots m_{k-1}}$ have the effect of producing the \ll expected\rr pairing matrix
\[
\left(
\begin{array}{c|ccc}
\rule{0ex}{2.6ex}
\vphantom{M_{M_{M_M}}}  
 M_1^{-1} & & &  \\
\hline
\rule{0ex}{2.6ex}
 & & & M_2^{-1} \\ 
 \rule{0ex}{2.6ex}
  & & \iddots & \\ 
  \rule{0ex}{2.6ex}
   & M_{k}^{-1}   
\end{array}
\right),
\quad
M_i=
\begin{pmatrix}
 & & m_i \\
 & \iddots & \\
 m_i & &
 \end{pmatrix}.
\]
Furthermore, let us assume that each $Q_i$ is homogeneous and polynomial in $q^{1/l}$. Then the differential equation again reduces to a system of algebraic equations for $Q_1, Q_2,\dots$ and it is easy to show that there is a unique solution.  

We shall use the above basis $P_0,\dots,P_{\s-1}$ to construct in turn a product operation, a grading, and a pairing.

\noindent{\em 1. The product}

Let us group the basis elements of $\calM$ as follows:
\begin{gather*}
(\h\b)^i\quad\text{for}\quad 0\le i\le u_1-1
\\
(\h\b)^i m_1 q^{-\De_1}(\h\b)^{u_1}     
\quad\text{for}\quad 0\le i\le u_2-1
\\
\dots
\\
(\h\b)^i 
m_1 q^{-\De_1}(\h\b)^{u_1} 
\dots\ 
m_{k-1} q^{-\De_{k-1}}(\h\b)^{u_{k-1}-1}     
\quad \text{for}\quad
0\le i\le u_k-1
\end{gather*}
Replacing $\h\b$ by $p$ here, and introducing the notation
\[
\one_{f_{i+1}}=
m_1\dots m_i  q^{-\De_1-\cdots-\De_i} 
p^{u_1+\cdots+u_i}
\]
we obtain a corresponding basis 
\begin{gather*}
1,\  
p,\ 
\dots,\ 
p^{u_1-1}
;
\\
\one_{f_2},\ \  
\one_{f_2} p,\ \ 
\dots,\ \ 
\one_{f_2} p^{u_2-1};
\\
\vdots\\
 \one_{f_k},\ \
 \one_{f_k} p,\ \ 
 \dots, \ \
 \one_{f_k} p^{u_k -1}
\end{gather*}
of $QA$. The vector space spanned (over $\C$) by these basis elements will be denoted $A$.
By definition, the action of $p$ on $A\otimes\C[q^{\pm1/l}]$ is given (with respect to this basis) by the matrix of Theorem
\ref{matrix}.  As $1$ is a cyclic element, this action extends to a product operation on $A\otimes\C[q^{\pm1/l}]$, that is, it allows us to define the product of any two elements
$\one_{f_i}p^j,   \one_{f_k}p^l$.  
We denote this product by
$\one_{f_i}p^j  \circ   \one_{f_k}p^l$, and regard $A\otimes\C[q^{\pm1/l}]$ as the abstract orbifold quantum cohomology ring of $\P(w)$.  We obtain a subring $A\otimes\C[q^{1/l}]$, and by putting $q=0$ we obtain a product operation on $A$, which we regard as the abstract orbifold cohomology.

\noindent{\em 2. The grading}

The differential operator $T_w -q$ is homogeneous of weight $2\s$,  if
we assign weights as follows:    
$\vert \h\vert=2$,
$\vert \b\vert=0$,
$\vert q\vert=2\s$.
The differential operators $P_0,\dots,P_{\s-1}$ are also homogeneous. Indeed, from the formula for  $P_{u_1+\cdots+u_i}$, its weight is
\begin{align*}
\vert P_{u_1+\cdots+u_i}\vert &= 
2(u_1+\cdots+u_i) - 2\s(\De_1+\cdots+\De_i)\\
&=
2(u_1+\cdots+u_i) - 2\s f_{i+1}\\
&=
2\age \one_{f_{i+1}}.
\end{align*}
It follows that our product operation satisfies
\[
\vert \one_{f_i}p^j  \circ   \one_{f_k}p^l \vert=
\vert \one_{f_i}p^j\vert +\vert  \one_{f_k}p^l \vert
\]
and $\vert\ \ \vert$ coincides with the usual orbifold quantum cohomology grading.

\noindent{\em 3. Self-adjointness and the pairing}

We shall obtain a natural identification of the $D^\h$-module 
$\calM=D^\h/(T_w -q)$ with a \ll dual\rr $D^\h$-module;  this will give us a  pairing on $\calM$, and a nondegenerate symmetric $\C[q^{\pm1/l}]$-linear pairing on $A\otimes\C[q^{\pm1/l}]$.  This pairing will turn out to be a $\C[q^{\pm1/l}]$-linear extension of a $\C$-linear pairing on $A$. We shall use the notation of section 6.3 of 
\cite{Gu08}.  

First, the $D^\h$-module $\calM^\ast$ is defined to be the space  of $\calF$-module homomorphisms $\calM\to \calF$, where $\calF$ is the coefficient ring. The $D^\h$-module structure is given by
\[
(\h\cdot \pi)(P) = \h\pi(P),\quad
(\b\cdot \pi)(P) = -\pi(\b\cdot P) +
q\tfrac{\b}{\b q}\pi(P)
\]
for $\pi\in \calM^\ast$.  

Next, we denote by $\bar{\calM}^\ast$ the $D^\h$-module obtained from $\calM^\ast$ by reversing the sign in the action of $\h$.  That is, $\bar{\calM}^\ast=\calM^\ast$ 
(as $\calF$-modules), but with action of $D^\h$ derived in the obvious way from
$\h\odot \pi =-\h \pi$, 
$\b\odot \pi =
\b\cdot \pi$.

Let $\de_0,\dots,\de_{\s-1}$ be the basis of $\bar{\calM}^\ast=\calM^\ast$ (over $\calF$) which is dual to the basis $1,\h\b,\dots,(\h\b)^{\s-1}$ of $\calM$.  The key technical result we need is:

\begin{proposition}\label{sa}  $ $

\noindent (1)  $\de_n$ is a cyclic element of
$\bar{\calM}^\ast$ (that is, $D^\h\odot\de_n=
\bar{\calM}^\ast$).

\noindent(2) $(T_w -q)\odot\de_n=0$. 
 
\noindent(3) The map $\calM\to \bar{\calM}^\ast$,
$[P]\mapsto [P\odot\de_n]$ is an isomorphism
of $D^\h$-modules.
\end{proposition}

\noindent It should be noted that the operator $T_w -q$ is
self-adjoint only in the special case $\P(w)=\C P^n$, even though $\calM=D^\h/(T_w -q)$
is always a self-adjoint $D^\h$-module.

\begin{proof}   Let $P_0^\ast,\dots,P_{\s-1}^\ast$ be the basis of $\bar{\calM}^\ast$ which is dual to $P_0,\dots,P_{\s -1}$.  For readability we shall omit square brackets throughout this proof.  Note that $P_i^\ast=\de_i$ for $i=0,\dots,n$.

We claim that
\[
P_\al\odot\de_n=
\begin{cases}
P^\ast_{n-\al}=\de_{n-\al}\ \text{when $0\le\al<u_1=n+1$},
\\
\tfrac{m_1}{m_{i+1}}  
P^\ast_{\s+n-\al}\ \text{when $u_1+\cdots+u_i\le\al<u_1+\cdots+u_{i+1}$},
\\
P^\ast_n=\de_n \ \text{when $\al=\s$ (we define $P_s$ below)}.
\end{cases}
\]
Assuming this, the first two formulae (for $\al=0,\dots,\s-1$) prove (1).   
In the third formula $P^\ast_n=\de_n$,  $P_\s$ means 
$\tfrac{1}{r_\s} \h\b 
\tfrac{1}{r_{\s-1}}  \h\b 
\dots
\tfrac{1}{r_1} \h\b$, which is $q^{-1}T_w$,
so this gives (2).  The third statement is an immediate consequence of (1) and (2) (cf.\  section 6.3 of 
\cite{Gu08}).

To prove the claim, we shall make use of
\begin{gather*}
\h\b P_\al = r_{\al+1} P_{\al+1}
\quad\quad\quad\quad(\ast)
\\
\h\b \odot P^\ast_\al = r_\al P^\ast_{\al-1}
\quad\quad\quad\quad(\ast\ast)
\end{gather*}
and the value of $r_\al$ given in Definition \ref{r}.

\noindent{\it The case $0\le \al < u_1=n+1$.}  

Since $r_0=\cdots=r_{n}=1$,  from $(\ast\ast)$ we have
$P_\al\odot\de_n = (\h\b)^\al \odot P^\ast_n =
P^\ast_{n-\al}$.

\noindent{\it The case $u_1+\cdots+u_i\le\al<u_1+\cdots+u_{i+1}$.}  

We shall prove this by induction on $i=0,1,\dots,k-1$ (regarding  the previous case as $i=0$).  

\noindent(i) If $\al=u_1+\cdots+u_i$ for some $i\ge 1$, we have
\begin{align*}
P_\al \odot\de_n &= m_i q^{-\De_i} \h\b P_{\al-1}\odot \de_n
\quad\text{by $(\ast)$, as $r_\al=m_i^{-1} q^{\De_i}$}
\\
&= 
m_i q^{-\De_i} \h\b \odot \tfrac{m_1}{m_i} 
P^\ast_{\s+n-(\al-1)}
\quad\text{(inductive hypothesis)}
\\
&=
m_1 q^{-\De_i}r_{\s + n -\al + 1}P^\ast_{\s+n-\al}
\quad\text{by $(\ast\ast)$}.
\end{align*}
Now, $\s+n-\al+1=\s + u_1 -(u_1+\cdots+u_i)
=\s - (u_{k+2-i}+\cdots+u_k)$ (by  Lemma \ref{fandu})
$= u_1+\cdots+u_{k+1-i}$.  (This argument applies only if $i\ge 2$, but the case $i=1$ is obvious.) Hence
\[
r_{\s + n -\al + 1}=r_{u_1+\cdots+u_{k+1-i}}
=\tfrac 1 {m_{k+1-i} } q^{\De_{k+1-i} }
= \tfrac 1 {m_{i+1} } q^{\De_{i} }
\]
by  Lemma \ref{Deandm}.  We obtain
$P_\al\odot \de_n =\tfrac{m_1}{m_{i+1}} P^\ast_{\s + n-\al}$.

\noindent(ii) If $u_1+\cdots+u_i\le\al<u_1+\cdots+u_{i+1}$ for some $i$, then
\begin{align*}
P_\al \odot\de_n &= \h\b P_{\al-1}\odot \de_n
\quad\text{by $(\ast)$, as $r_\al=1$}
\\
&= 
\h\b \odot \tfrac{m_1}{m_{i+1}} 
P^\ast_{\s+n-(\al-1)}
\quad\text{(inductive hypothesis)}
\\
&=
\tfrac{m_1}{m_{i+1}}  r_{\s + n -\al + 1}P^\ast_{\s+n-\al}
\quad\text{by $(\ast\ast)$}.
\end{align*}
Here we have 
$\s+n-\al+1=u_1+\cdots+u_{k+1-i} - l$
with $0<l<u_{i+1} = u_{k+1-i}$ (from Lemma \ref{fandu}),
so $r_{\s + n -\al + 1}=1$. 
We obtain
$P_\al\odot \de_n =\tfrac{m_1}{m_{i+1}} P^\ast_{\s + n-\al}$ again.

\noindent{\it The case $\al=\s$.}  

We have
\begin{align*}
P_\s \odot\de_n &= m_k q^{-\De_k} \h\b P_{\s-1}\odot \de_n
\quad\text{by $(\ast)$}
\\
&= 
m_k q^{-\De_k} \h\b \odot \tfrac{m_1}{m_k} 
P^\ast_{n+1}
\quad\text{(inductive hypothesis)}
\\
&=
m_1 q^{-\De_k}r_{n + 1}P^\ast_{n}
\quad\text{by $(\ast\ast)$}.
\end{align*}
Here we have $r_{n+1}=r_{u_1}=\tfrac1{m_1} q^{\De_1}$, and $\De_1=\De_k$ by Lemma \ref{Deandm}, so we conclude that $P_\s\odot \de_n=\de_n$.
\end{proof}

The natural composition
$\calM\times\calM\to
\bar{\calM}^\ast\times \calM\to \calF$, making use of the above isomorphism $\calM\to \bar{\calM}^\ast$, defines a pairing.  We normalize it as follows:

\begin{definition}\label{pairing}
$\lann P , Q\rann=
\tfrac1{w_0\dots w_n}(P\odot\de_n)(Q)
\ (=
\tfrac1{m_1}(P\odot\de_n)(Q)
)$.
\end{definition}

\begin{corollary}\label{metric} We  have (from the formula for 
$P_\al\odot\de_n$ in the proof of Proposition \ref{sa})
\[
\lann P_\al ,P_\be \rann=
\begin{cases}
\tfrac1{m_1}\de_{n-\al,\be}\ 
\text{if $0\le\al<u_1$},
\\
\tfrac{1}{m_{i+1}}  
\de_{\s+n-\al,\be}\ \text{if $u_1+\cdots+u_i\le\al<u_1+\cdots+u_{i+1}, i\ge 1$}.
\end{cases}
\]
\end{corollary}

With this normalization,
the induced pairing on $A$ agrees with the usual Poincar\'e intersection pairing on the cohomology of $\P(w)$; it is known from \cite{Ka73} that $(1,p^n)=1/(w_0\dots w_n)$.  The induced pairing on 
 $A\otimes \C[q^{\pm1/l}]$ satisfies the Frobenius property
(see section 6.5 of  \cite{Gu08}).  Hence, by the cyclic property, it agrees with the orbifold quantum Poincar\'e intersection pairing. 
 
 This concludes our construction of an abstract orbifold quantum product, grading, and pairing directly from $T_w-q$, and our verification that they agree with the usual ones.

\begin{example}\label{example}  $\P(1,2,3)$

We have $w_0=1, w_1=2, w_2=3$ and $\s=1+2+3=6$.
The differential operator is
\begin{align*}
T_w-q&= \h\b\ 2\h\b(2\h\b-\h) \
3\h\b(3\h\b-\h)(3\h\b-2\h) -q
\\
&=2^23^3\h^6\b^3(\b-\tfrac13)(\b-\tfrac12)(\b-\tfrac23)-q.
\end{align*}
This has order $6$, and it is homogeneous of weight 
$12$, where $\vert\h\vert=2$, $\vert q\vert=12$.

We have
$F=\{ \frac01, \frac02, \frac12, \frac03, \frac13, \frac23 \}
=
\{ 0, \frac13,  \frac12,  \frac23 \}$,   so
$u_1=3, u_2=1, u_3=1, u_4=1$.  It is convenient to display all relevant data in the following diagram:
\[
\begin{array}{r|c|c|c|c}
\vphantom{\dfrac AA}
 &w_0=1 & w_1=2 & w_2=3
\\
\hline
\vphantom{\dfrac AA}
S_{f_1}=\{0,1,2\},\  f_1=0 
& \frac01  &  \frac02  &  \frac03  
&\De_1=\frac13, m_1=6
\\
\hline
\vphantom{\dfrac AA}
S_{f_2}=\{2\},\  f_2=\frac13 
& & & \frac13
&\De_2=\frac16, m_2=3
 \\
 \hline
 \vphantom{\dfrac AA}
S_{f_3}=\{1\}, \ f_3=\frac12 
&  & \frac12 & 
&\De_3=\frac16, m_3=2
  \\
  \hline
  \vphantom{\dfrac AA}
S_{f_4}=\{2\}, \ f_4=\frac23 & & & \frac23
&\De_4=\frac13, m_4=3
 \\
 \hline
\end{array}
 \]
 \[
 {}
 \]
In the central $4\times 3$ block,  the number of entries in the $i$th row is $u_i$,  and the number of entries in the $j\!\!+\!\!1$th column is $w_{j}$.
 
 The factorization is
 \[
q^{-1} T_w=
3 q^{-\frac13} (\h\b)^{1}
2 q^{-\frac16} (\h\b)^{1}
3 q^{-\frac16} (\h\b)^{1}
6 q^{-\frac13} (\h\b)^{3}.
\]
The bases of $\calM$ and $A$ constructed above are:
\begin{align*}
&1, \h\b, (\h\b)^2
&
&1,\ p,\ p^2
\\
&6 q^{-\frac13} (\h\b)^3
&
&\one_{\frac13}
\\
&3 q^{-\frac16} (\h\b) \  6 q^{-\frac13} (\h\b)^3
&
&\one_{\frac12}
\\
&2 q^{-\frac16} \h\b \  3 q^{-\frac16} \h\b \ 6 q^{-\frac13} (\h\b)^3
&
&\one_{\frac23}
\end{align*}
The matrix of structure constants (quantum multiplication by $p$) with respect to this basis is
\newcommand{\hp}{ \hphantom{qq} }
\[
\begin{pmatrix}
\ 0\  &  &  &  &  & \!\frac13q^{\frac13}
 \\
1  & \ 0\  & & & &
 \\
  & 1 & 0 & & &
 \\
  & & \frac16 q^{\frac13} & 0 & &
 \\
  & & & \frac13 q^{\frac16} &0 &
 \\
  & & & & \frac12 q^{\frac16} & 0
\end{pmatrix}.
\]
These products 
determine all others, and we obtain the following orbifold quantum multiplication table:
\renewcommand{\AA}{\tfrac16 q^{\tfrac {A^A}A} \one_{\tfrac2{A_A}}}
\[
\begin{array}{c|cccccc}
\vphantom{\AA}
 & \ \ 1\ \  & \ \ p\ \  & p^2 & \one_{\frac13}
 & \one_{\frac12} & \one_{\frac23}
\\
\hline
\vphantom{\AA}
1  & 1 & p & p^2 & \one_{\frac13}
 & \one_{\frac12} & \one_{\frac23}
 \\
\vphantom{\AA}
p  &  & p^2 & \frac16 q^{\frac13} \one_{\frac13} & \frac13 q^{\frac16} \one_{\frac12} & \frac12 q^{\frac16} \one_{\frac23} & \frac13 q^{\frac13}  
\\
\vphantom{\AA}
p^2  &  &  & \frac1{18} q^{\frac12} \one_{\frac12} & \frac16 q^{\frac13} \one_{\frac23} & \frac16 q^{\frac12}  & \frac13 q^{\frac13}  p
\\
\vphantom{\AA}
\one_{\frac13}  &  &  &  & \frac13 q^{\frac13}& 
q^{\frac16}p  & 2p^2
\\
\vphantom{\AA}
\one_{\frac12}  &  &  &  & & 
3p^2  & q^{\frac16}\one_{\frac13}
\\
\vphantom{\AA}
\one_{\frac23}  &  &  &  & & 
  & \frac23 q^{\frac16}\one_{\frac12}
\\
\end{array}
 \]
Orbifold cohomology products are obtained by setting $q=0$ in this table.  Note that $p$ generates the orbifold quantum cohomology, but not the orbifold cohomology.
Ages and degrees are as shown below:
\[
\begin{array}{r|c|c|c|}
\hline
\vphantom{\dfrac AA}
\age \one_{0} = 0
& \vert 1\vert=0  &   \vert p\vert=2  &   \vert p^2\vert=4  
\\
\hline
\vphantom{\dfrac AA}
\age \one_{\frac13} = 1
& & &  \vert \one_{\frac13}\vert=2
 \\
 \hline
 \vphantom{\dfrac AA}
\age \one_{\frac12} = 1
&  & \vert \one_{\frac12}\vert=2 & 
  \\
  \hline
  \vphantom{\dfrac AA}
\age \one_{\frac23} = 1
& & & \vert \one_{\frac23}\vert=2
 \\
 \hline
\end{array}
 \]
Finally, the pairing on $\calM$ is given by
$\lann P_i, P_j \rann=\tfrac16$ if $i+j=2$,
$\lann P_3, P_5 \rann= \lann P_5, P_3 \rann=\tfrac13$,
and
$\lann P_4, P_4 \rann=\tfrac12$
(with all other products zero).  
\qed
\end{example}

\begin{example}  $\P(1,1,3)$

In this case we have orbifold classes with fractional degrees. We just state the results, as the calculations are very similar to those in the previous example.   First, the data is
\[
\begin{array}{r|c|c|c|c}
\vphantom{\dfrac AA}
 &w_0=1 & w_1=1 & w_2=3
\\
\hline
\vphantom{\dfrac AA}
S_{f_1}=\{0,1,2\},\  f_1=0 
& \frac01  &  \frac01  &  \frac03  
&\De_1=\frac13, m_1=3
\\
\hline
\vphantom{\dfrac AA}
S_{f_2}=\{2\},\  f_2=\frac13 
& & & \frac13
&\De_2=\frac13, m_2=3
 \\
 \hline
 \vphantom{\dfrac AA}
S_{f_3}=\{2\}, \ f_3=\frac23
&  &  & \frac23
&\De_3=\frac13, m_3=3
  \\
 \hline
\end{array}
 \]
 \[
 {}
 \]
and we have
\[
q^{-1}T_w= 
q^{-1}3^3\h^5\b^3(\b-\tfrac13)(\b-\tfrac23)
=
3 q^{-\frac13} (\h\b)^1
3 q^{-\frac13} (\h\b)^1
3 q^{-\frac13} (\h\b)^{3}.
\]
The orbifold quantum multiplication table is
\renewcommand{\AA}{\tfrac16 q^{\tfrac {A^A}A} \one_{\tfrac2{A_A}}}
\[
\begin{array}{c|ccccc}
\vphantom{\AA}
 & \ \ 1\ \  & \ \ p\ \  & p^2 & \one_{\frac13}
  & \one_{\frac23}
\\
\hline
\vphantom{\AA}
1  & 1 & p & p^2 & \one_{\frac13}
 &  \one_{\frac23}
 \\
\vphantom{\AA}
p  &  & p^2 & \frac13 q^{\frac13} \one_{\frac13} & \frac13 q^{\frac13} \one_{\frac23} & \frac13 q^{\frac13}  
\\
\vphantom{\AA}
p^2  &  &  & \frac1{9} q^{\frac23} \one_{\frac23} & \frac19 q^{\frac23}  & \frac13 q^{\frac13}p  
\\
\vphantom{\AA}
\one_{\frac13}  &  &  &  & \frac13 q^{\frac13}p& 
p^2
\\
\vphantom{\AA}
\one_{\frac23}  &  &  &  & & 
\one_{\frac13}
\\
\end{array}
 \]
 where $1, p, p^2, \one_{\frac13}, \one_{\frac23}$
 correspond to $1, \h\b, (\h\b)^2, 
3q^{-\frac13} (\h\b)^3,
3q^{-\frac13} (\h\b)  3q^{-\frac13} (\h\b)^3$.
We have
\[
\begin{array}{r|c|c|c|}
\hline
\vphantom{\dfrac AA}
\age \one_{0} = 0
& \vert 1\vert=0  &   \vert p\vert=2  &   \vert p^2\vert=4  
\\
\hline
\vphantom{\dfrac AA}
\age \one_{\frac13} = \frac43
& & &  \vert \one_{\frac13}\vert=\frac83
  \\
  \hline
  \vphantom{\dfrac AA}
\age \one_{\frac23} = \frac23
& & & \vert \one_{\frac23}\vert=\frac43
 \\
 \hline
\end{array}
 \]
and the pairing  is given by
$\lann P_i, P_j \rann=\tfrac13$ if $i+j=2$,
$\lann P_3, P_4 \rann= \lann P_4, P_3 \rann=\tfrac13$
(with all other products zero). 
\qed
\end{example}

\section{Hypersurfaces in weighted projective space}\label{hyper}

Based on the toric approach to mirror symmetry, 
Corti and Golyshev conjectured that the orbifold quantum cohomology of a (quasismooth) hypersurface
\[
X^d\sub \PW
\]
of degree $d$ is governed by the differential operator 
\[
\h^\s \prod_{i=0}^n (w_i \b) (w_i \b - 1)\dots
(w_i \b - (w_i-1))\  - \ 
q\h^d  (d\b+1)\dots (d\b+(d-1))(d\b+d).
\]
(this operator appears in section 7.3 of \cite{GoXX} without the $\h$ factors; also in \cite{CoGoXX} for the 
Calabi-Yau case $\s=d$, where the $\h$ factors cancel out).  The method of \cite{CCLTXX} gives evidence for this conjecture in the Fano case, i.e.\  when $\s>d$. We shall always assume that $\s>d$, although our approach applies also when $\s=d$ (cf.\ section 6.7 of \cite{Gu08}).

Since $\b q=q(\b + 1)$, we have
\[
q\h^d  d^d (\b+\tfrac1d)\dots (\b+\tfrac{d-1}d)(\b+\tfrac{d}d)=
\h^d  d^d (\b-\tfrac{d-1}d)\dots (\b-\tfrac1d)(\b-\tfrac0d) q,
\]
which shows that both summands of
\[
w^w \h^\s 
\prod_{i=0}^n  
\b(\b-\tfrac{1}{w_i})\dots(\b - \tfrac{w_i-1}{w_i})
\ -\ 
q
\h^d  d^d (\b+\tfrac1d)\dots (\b+\tfrac{d-1}d)(\b+\tfrac{d}d)
\]
can be written with a factor of $\h\b$ on the left. Cancelling this factor, we obtain an operator of order $\s-1$ (in terms of $D^\h$-modules, we quotient out by the trivial $D^\h$-module $D^\h/(\h\b)$). We call\footnote{We do not write $T_w$ here; the abbreviation 
$T_w$ always means $T_{w_0,\dots,w_n}$.
}
this
operator
$T_{w_1,\dots,w_n}-qS_{d-1}$:
\[
\underbrace{
w^w \h^{\s-1}
\prod_{i=1}^n  
\b(\b-\tfrac{1}{w_i})\dots(\b - \tfrac{w_i-1}{w_i})
}_{T_{w_1,\dots,w_n} }
\ -\ 
q
\underbrace{
\vphantom{
\prod_{i=1}^n  
\b(\b-\tfrac{1}{w_i})\dots(\b - \tfrac{w_i-1}{w_i})
}
\h^{d-1}  d^d (\b+\tfrac1d)\dots (\b+\tfrac{d-1}d)
}_{S_{d-1} }
\]
Here we have assumed that $w_0=1$.  {\em To simplify notation, we shall also assume that $w_1,\dots,w_n$ are such that no further left-cancellations of the above type are possible.}  It follows that the $D^\h$-module
\[
\calM=D^\h/(T_{w_1,\dots,w_n}-qS_{d-1})
\]
is irreducible.  In the general case, an irreducible $D^\h$-module is obtained by left-cancelling all common factors (see \cite{GoXX}), and our method can
be applied to that.

Observe that the case $d=1$ gives $T_{w_1,\dots,w_n}-q$, which is the operator associated with $\P(w_1,\dots,w_n)$, as expected. The case $w_1=\cdots=w_n=1$ (hence $s=n+1$) gives $(\h\b)^{n}-qS_{d-1}$, which is the operator associated with a degree $d$ hypersurface in $\C P^{n}$, denoted by $M^d_{n+1}$ in \cite{Sa08}.

In this section, by extending the method of section \ref{direct},
we shall give a method to extract an \ll abstract orbifold quantum product\rrr.  We emphasize that this is a method whose success is not guaranteed.  Moreover, the (genuine) quantum product is not yet known for hypersurfaces, in general.  Nevertheless, we can give a nontrivial example (Example \ref{corti}) where the Gromov-Witten invariants have been proposed by Corti (\cite{Co08}), and our method is consistent with his results.   

As in section \ref{notation}, we define
\begin{align*}
F&=\{  \tfrac{i}{w_j} \st  0\le i \le w_j-1, \ \ 
1\le j\le n \}\\
&=\{  
f_1,\dots,f_k
\}.
\end{align*}
and denote by $u_1,\dots,u_k$  the multiplicities of $f_1,\dots,f_k$.  However,  $u_1=n$ here. We use the notation $\De_i, m_i$ as in section \ref{results}.  Thus, we have a factorization
\[
q^{-1}T_{w_1,\dots,w_n}=
\tfrac{1}{r_{\s-1}} \ \h\b\ 
\tfrac{1}{r_{\s-2}} \ \h\b\ 
\ \ \dots\ \ 
\tfrac{1}{r_1} \h\b\ 
\]
and we can introduce $P_0=1$ and
$P_i=\frac{1}{r_i} \h\b P_{i -1}$ for $1\le i \le \s-2$.  
The equivalence classes of the operators $P_0,\dots,P_{\s-2}$ form a basis of the $D^\h$-module $D^\h/(T_{w_1,\dots,w_n}-qS_{d-1})$. 

As in section \ref{direct}, the action of $\b$ defines a connection on the bundle whose space of sections is $\calM$.  However, when $d\ge 2$, the connection matrix $\Om$ with respect to the basis $P_0,\dots,P_{\s-2}$ is not of the form $\tfrac1\h \om$.  To achieve this form (which is the starting point for the construction of a product operation) we must construct a new basis.  

It will be convenient to construct such a basis in two steps.  

\noindent{\em Step 1}\ \   The method of \cite{Sa08} produces a basis
$\hat P_0,\dots,\hat P_{\s-2}$ with respect to which the connection matrix has the form $\tfrac1\h \hat\om$.   Let us review that method here. As in our discussion of the Birkhoff factorization method in section \ref{direct}, the new basis is given by 
$L_+^{-1}\cdot P_0, L_+^{-1}\cdot P_1,\dots, L_+^{-1}\cdot
P_{\s-2}$, for a certain \ll gauge transformation\rr
$L_+=Q_0(I + \h Q_1 + \cdots )$.  In contrast to the situation of section \ref{direct}, there is no short cut to finding $L_+$ here.  However, $L_+$ can be found as the unique solution of the ordinary differential equation
\[
\tfrac1\h \hat\om = L_+ \Om L_+^{-1} + L_+ dL_+^{-1}
\]
which is homogeneous and polynomial in $q^{1/l}$, and
which satisfies the initial condition $L_+\vert_{q^{1/l}=0}=I$.  As in \cite{Sa08} it can be proved that this
reduces to a system of algebraic equations for $Q_0,Q_1,\dots$ which can be solved by an explicit algorithm.   (In the situation of section \ref{direct} we would have $L_+=I$, as we are now starting from the basis given by $P_i=\tfrac{1}{r_i}\h\b\cdots\tfrac{1}{r_1}\h\b$. The $L_+$ of section \ref{direct} 
converts $1,\h\b,\dots,(\h\b)^{s-2}$ to $P_0,\dots,P_{s-2}$, which could have been described as Step 0.)

\begin{example}  $X^3\sub\P(1,1,1,1,1)=\C P^4$

In the notation of \cite{Sa08} this is $M^3_5$. As this example is worked out in detail in Examples 3.6, 5.4, 6.24, 6.36 of \cite{Gu08} we shall just summarize the results of the calculations.

First, we have the differential operator
\[
q^{-1}T_{1,1,1,1} - S_2 =
q^{-1} (\h\b)^4 - 3^3 \h^2 (\b+\tfrac13)(\b + \tfrac23).
\]
With respect to the basis $P_0=1, P_1=\h\b, P_2=(\h\b)^2, P_3=(\h\b)^3$ the connection matrix is
\[
\Om=
\tfrac1\h
\begin{pmatrix} 
 & & & 6q\h^2 \\
 1 & & & 27q\h \\
  & 1 & & 27q \\
   & & 1 & 
\end{pmatrix}.
\]
The gauge transformation $L_+=Q_0(I+\h Q_1)$  can be found by solving the o.d.e.\
$
\tfrac1\h \hat\om = L_+ \Om L_+^{-1} + L_+ dL_+^{-1}
$
subject to $L_+\vert_{q=0}=I$.  This gives
\[
Q_0=
\begin{pmatrix} 
\, 1\,  & &6q &  \\
  & \,1\,& & 21q \\
  &  & 1&  \\
   & &  & 1
\end{pmatrix},
\quad
Q_1=
\begin{pmatrix} 
\ \  & \ \ & \ \ &  6q\\
  & & &  \\
  &  & &  \\
   & &  & 
\end{pmatrix}.
\]
The new basis is $\hat P_0=1, \hat P_1= \h\b, \hat P_2= (\h\b)^2-6q,
\hat P_3= (\h\b)^3 - 21q\h\b - 6\h q$, and the matrix of the connection form is
\[
\hat\Om=
\tfrac1\h
\begin{pmatrix} 
 &6q & & 36q^2 \\
 1 & &15q & \\
  & 1 & & 6q \\
   & & 1 & 
\end{pmatrix}
\]
with respect to this basis. 
\qed
\end{example}

The basis $\hat P_0,\dots,\hat P_{\s-2}$ allows us to construct a product operation as in section \ref{direct}.  Unfortunately, this product does not necessarily satisfy the Frobenius property. In general, therefore, it is necessary to modify the basis further, and this will be Step 2.  As preparation for this, we begin with a brief review of the Birkhoff decomposition.  

\noindent{\em Motivation for Step 2}\ \   The Birkhoff decomposition (Theorem 8.1.2 of \cite{PrSe86}) can be written
\[
\La GL_{\s-1}\C =
\bigcup_{\ga\in\check T}
\La_- GL_{\s-1}\C \ \ga\ \La_+ GL_{\s-1}\C,
\]
where $\check T$ denotes the set of homomorphisms from $S^1$ to the diagonal matrices in $ GL_{\s-1}\C$. If $\ga(\h)=\diag(\h^{a_0},\dots,\h^{a_{s-2}})$
is restricted to the set of homomorphisms satisfying $a_0\le \cdots\le a_{\s-2}$, then the decomposition is a disjoint union. The \ll big cell\rr is the piece given by $\ga=I$; it is a dense open subset of the identity component of $\La GL_{\s-1}\C$. The \ll small cells\rr (where $\ga\ne I$) have finite codimension in 
$\La GL_{\s-1}\C$.  

The term \ll cell\rr is used here because the decomposition is equivalent to the $\La_- GL_{\s-1}\C$-orbit decomposition
\[
\La GL_{\s-1}\C/ \La_+ GL_{\s-1}\C =
\bigcup_{\ga\in\check T}
\La_- GL_{\s-1}\C \ [\ga]
\]
of the Grassmannian
$Gr^{(\s-1)} \cong 
\La GL_{s-1}\C/ \La_+ GL_{\s-1}\C$ (see section 8.3 of \cite{PrSe86}). It is analogous to the cell decomposition, or cell-bundle\footnote{The cell decompositions here arise from Morse functions; the cell-bundle decompositions arise from Morse-Bott functions.}
 decomposition, of a finite-dimensional generalized flag manifold given by the orbits of a parabolic subgroup.
 The main point is that a \ll small cell\rr $\La_- GL_{\s-1}\C \ [\ga]$ is diffeomorphic to a proper unipotent subgroup $\La^\ga_-$ of
$\La_- GL_{\s-1}\C$ (Theorem 8.6.3 of \cite{PrSe86}). 
This shows that any map $L$ which takes values in 
$\La_- GL_{\s-1}\C \ \ga\ \La_+ GL_{\s-1}\C$ (and therefore admits at least one factorization $L=L_-\ga L_+$) has a {\em most economical} factorization
\[
L=L^c_-\ \ga\  L^c_+.
\]
The same phenomenon occurs for finite-dimensional generalized flag manifolds. The simplest example is $\C P^n$: the $i$-dimensional cell $\C^i$ can be described as an orbit of the 
$\tfrac12(n+1)(n+2)$-dimensional group of upper triangular matrices in $GL_{n+1}\C$, but most economically as an orbit of a certain $i$-dimensional unipotent subgroup (see chapter 14, part III, of 
\cite{Gu97}). 

Step 2 will amount to extracting the economical factor $L^c_-$ from $L_-\ga$.  More precisely, by Theorem 8.6.3 of \cite{PrSe86}, we can write
\[
L=L_-\ga L_+
= L^c_-L^f_- \ga L_+
=L^c_-\ga L^f_+  L_+
\]
where $L^f_-$ denotes the \ll superfluous factor\rrr; this is a polynomial in $\h^{-1}$ and satisfies
$L^f_- \ga = \ga L^f_+$ where $L^f_+$ is polynomial in $\h$. 
Thus, Step 1 uses the gauge transformation $L_+^{-1}$ to convert $P_0,\dots,P_{\s-2}$ to a provisional basis $\hat P_0,\dots,\hat P_{\s-2}$, then Step 2 uses a further gauge transformation $(\ga L_+^f)^{-1}$ to convert $\hat P_0,\dots,\hat P_{\s-2}$ to the desired basis $\tilde P_0,\dots,\tilde P_{\s-2}$.

\noindent{\em Step 2}\ \    As in Proposition \ref{sa}, it can be proved that 
\[
(T_{w_1,\dots,w_n}-qS_{d-1})\odot\de_{n-1}=0.
\]
However the map 
\[
\calM\to \bar{\calM}^\ast,\quad
[P]\mapsto [P\odot\de_{n-1}]
\]
is not in general an isomorphism
of $D^\h$-modules, and it is at this point that we need the homomorphism $\ga$.  Let us  assume that

\noindent (H1) there exist integers $a_0\le \dots\le a_{s-2}$ with the property that 
$
\h^{-a_0} \hat P_0 \odot \de_{n-1},\dots,
\h^{-a_{s-2}} \hat P_{s-2} \odot \de_{n-1}
$
have {\em minus} the weighted degrees of the elements
$\h^{-a_0}\hat P_0,\dots,\h^{-a_{s-2}}\hat P_{s-2}$ (not necessarily in the same order).  

\noindent As in Definition \ref{pairing}, we can define a pairing by
\[
\lann P , Q\rann=
\tfrac d{w_1\dots w_n}(P\odot\de_{n-1})(Q).
\]
The normalization of section \ref{direct} is modified by the factor $d$ here, to take account of the degree of the hypersurface.
We shall assume further that

\noindent (H2) there exists a basis with respect to which the matrix of $\lann\ ,\ \rann$ is a nondegenerate symmetric matrix independent of $q,\hbar$.

\noindent That is, the pairing $\lann\ ,\ \rann$ extends a nondegenerate symmetric $\C$-linear pairing on a complex vector space of dimension $s-1$.
Finally (from the motivation above) we seek a map $G=L_-^f\ga = \ga L_+^f$ such that the gauge transformation $G^{-1}$ converts $\hat\Om$ to a connection form $\tilde\Om=\tfrac1\h \tilde\om$ where $\tilde\om$ is independent of $\h$. This $\tilde\om$ is the connection matrix with respect to the basis $\tilde P_0=G^{-1}\cdot \hat P_0,\dots,\tilde P_{s-2}=G^{-1}\cdot \hat P_{s-2}$. Such a $G$ necessarily satisfies

\noindent $(\ast)$
\quad\quad\quad\quad\quad\quad\quad\quad\quad
$\tfrac1\h \tilde\om = G \tfrac1\h \hat\om G^{-1} + G dG^{-1}$

\noindent and we shall assume that

\noindent (H3) there exists in some neighbourhood of $q^{1/l}=0$ a solution $G=L_-^f\ga$ of the differential equation $(\ast)$, where $\ga(\hbar)=\diag(\hbar^{a_0},\dots,\hbar^{a_{s-2}})$.

\noindent  We discuss suitable normalizations of such solutions (initial conditions) later.

Assuming (H1)-(H3) (which we shall verify in our main example), we can attempt to define a product operation as in section \ref{direct}. 
Let $A$ be the vector space with basis denoted by the symbols
\begin{gather*}
1,\  
p,\ 
\dots,\ 
p^{u_1-1}
;
\\
\one_{f_2},\ \  
\one_{f_2} p,\ \ 
\dots,\ \ 
\one_{f_2} p^{u_2-1};
\\
\vdots\\
 \one_{f_k},\ \
 \one_{f_k} p,\ \ 
 \dots, \ \
 \one_{f_k} p^{u_k -1}
\end{gather*}
We define $QA$ to be $A\otimes\C[q^{\pm 1/l}]$, and we define a $\C[p,q^{\pm 1/l}]$-module action on $QA$ by specifying that the matrix of multiplication by $p$ is $\tilde\om$.  As we shall see, in contrast to the situation of the previous section, this $\C[p,q^{\pm 1/l}]$-module action does not in general allow us to obtain a product structure on $QA$, because the action of $p$ is not necessarily cyclic.

\begin{example}  $X^3\sub\P(1,1,1,2)$\label{corti}

We have $w_0=w_1=w_2=1$, $w_3=2$ and $\s=5$, $d=3$.
The differential operator is
\[
q^{-1} T_{1,1,2} - S_2= 
q^{-1}2^2\h^4\b^3(\b-\tfrac12)  -
3^3\h^2(\b+\tfrac13)(\b+\tfrac23).
\]
We have
$F=\{\frac01, \frac01, \frac02, \frac12 \}
=
\{ 0,  \frac12\}$,   so
$u_1=3, u_2=1$.  As in section \ref{direct} we can display the data as follows:
\[
\begin{array}{r|c|c|c|c}
\vphantom{\dfrac AA}
 &w_1=1 & w_2=1 & w_3=2
\\
\hline
\vphantom{\dfrac AA}
S_{f_1}=\{1,2,3\},\  f_1=0 
& \frac01  &  \frac01  &  \frac02  
&\De_1=\frac12, m_1=2
\\
\hline
\vphantom{\dfrac AA}
S_{f_2}=\{3\},\  f_2=\frac12
& & & \frac12
&\De_2=\frac12, m_2=2
 \\
 \hline
\end{array}
 \]
 \[
 {}
 \]
The factorization of $q^{-1}T_{1,1,2}$  is
 \[
q^{-1}T_{1,1,2}=
2 q^{-\frac12} (\h\b)
2 q^{-\frac12} (\h\b)^{3}
=
\tfrac1r (\h\b)
\tfrac1r (\h\b)^{3},
\]
where $r=\frac12 q^{\frac12}$.
Thus, our starting point is the basis
\[
P_0=1,\ P_1=\h\b,\ P_2=(\h\b)^2,\ P_3=\frac1r(\h\b)^3.  
\]
We have $\vert r\vert=2$, so the degrees of these basis elements are $0,2,4,4$. With respect to this basis we have
\[
\Om = \tfrac1\h
\begin{pmatrix}
 & & & 6\h^2 r \\
1 & & & 27\h r\\
 & 1 & & 27 r\\
  & & r & 
\end{pmatrix}.
\]

\noindent{\em Step 1}\ \   The gauge transformation $L_+^{-1}$
is given by
\[
L_+=
\begin{pmatrix}
1 & & 12r^2 & \\
 & 1 & & 30r \\
  & & 1 & \\
\hphantom{q^{\frac12}}  & 
\hphantom{q^{\frac12}}& 
\hphantom{q^{\frac12}} & 1
   \end{pmatrix}
   \left(
   I +
   \h
   \begin{pmatrix}
    & & & 12r\\
     & & & \\
         & & & \\
\hphantom{q^{\frac12}}&
\hphantom{q^{\frac12}} &
\hphantom{q^{\frac12}} & 
             \end{pmatrix}
   \right).
\]
Application of $L_+^{-1}$
produces the new basis
\[
\hat P_0=P_0,\  \hat P_1=P_1,\ 
\hat P_2=P_2-12r^2P_0,\  
\hat P_3=P_3-30rP_1-12\h rP_0.
\]
With respect to this basis, we have
\[
\hat\Om = 
\tfrac1\h\hat\om=
\tfrac1\h
\begin{pmatrix}
 & 12r^2 & & -36r^3 \\
1 & & 18r^2 & \\
 & 1 & & -3r\\
  & & r & 
\end{pmatrix}.
\]
We omit the details of this calculation, which is similar to those in \cite{Sa08}.

\noindent{\em Step 2}\ \    We have to verify (H1)-(H3).  
For $\ga$ we take $\ga(\h)=(1,1,1,\h)$.  
The degrees of 
$\ga^{-1}\hat P_0, \ga^{-1}\hat P_1,
\ga^{-1}\hat P_2, \ga^{-1}\hat P_3$ are $0,2,4,2$, and the degrees of 
$(\ga^{-1}\hat P_0)\odot\de_2, (\ga^{-1}\hat P_1)\odot\de_2,
(\ga^{-1}\hat P_2)\odot\de_2, (\ga^{-1}\hat P_3)\odot\de_2$ are
$-4,-2,0,-2$, so (H1) is satisfied.  

To verify (H3), we note that $G^{-1}$ must be of the form
\[
G^{-1}=\ga^{-1}Z=\ga^{-1}(Z_0+\tfrac1\h Z_1+\tfrac1{\h^2} Z_2)
\]
where $Z$ is homogeneous with respect to a basis with degrees $0,2,4,2$,  i.e.\ the entries of the matrix function $Z$ have the degrees shown below:
\[
\begin{matrix}
\vphantom{\boxed{A^{A^A}}}
\boxed{\hphantom{-}0} & \boxed{\hphantom{-}2} & \boxed{\hphantom{-}4} & \boxed{\hphantom{-}2} \\
\vphantom{\boxed{A^{A^A}}}
\boxed{-2} & \boxed{\hphantom{-}0} & \boxed{\hphantom{-}2} & \boxed{\hphantom{-}0} \\
\vphantom{\boxed{A^{A^A}}}
\boxed{-4} & \boxed{-2} & \boxed{\hphantom{-}0} & \boxed{-2} \\
\vphantom{\boxed{A^{A^A}}}
\boxed{-2} & \boxed{\hphantom{-}0} & \boxed{\hphantom{-}2} & \boxed{\hphantom{-}0} 
\end{matrix}
\]
Equating the coefficients of each power of $\h$ in the above
differential equation $(\ast)$ gives a collection of equations for the coefficients of $Z_0,Z_1,Z_2$ and $\tilde\om$.  With the initial condition
$Z\vert_{r=0}=I$, the unique solution is
\[
Z=
\begin{pmatrix}
\ 1\  & & & \\
 & \ 1\  & & \\
  & & \ 1\  & \\
   & & -2r & 1
\end{pmatrix}
+
\tfrac1\h
\begin{pmatrix}
 & & & \\
\ \  & \ \  &-6r^2 &3r \\
  & &  & \\
   & &  & 
\end{pmatrix}
\]
More generally, the initial condition
$Z\vert_{q=0}=\diag(1,1,1,y)$ 
leads to the solution
\[
Z=
\begin{pmatrix}
\ 1\  & & & \\
 & \ 1\  & & \\
  & & \ 1\  & \\
   & & -2r & y
\end{pmatrix}
+
\tfrac1\h
\begin{pmatrix}
 & & & \\
\ \  & \ \  &-6r^2 &3ry \\
  & &  & \\
   & &  & 
\end{pmatrix}.
\]
The new basis ($\tilde P_i=G^{-1}\cdot \hat P_i$)
produced by Step 2 is, therefore,
\[
\tilde P_0=1,\ 
\tilde P_1=\h\b,\ 
\tilde P_2=\hat P_2 - \tfrac{2r}\h \hat P_3 - \tfrac{6r^2}\h\hat P_1,\  
\tilde P_3=\tfrac y{\h} \hat P_3 + \tfrac{3ry}\h \hat P_1.
\]
The connection matrix with respect to this basis is
\[
\tilde\Om=
\tfrac1\h\tilde\om=
\tfrac1\h
\begin{pmatrix}
 & 12r^2 & & \\
 1 & & 12r^2 & \frac{3yr}{2} \\
  & 1 & & \\
   & \frac{2r}y & &
   \end{pmatrix}.
\]
Finally we verify condition (H2) by explicit calculation of
 $\lann\ ,\ \rann$
 with respect to the new basis:
\[
\left(\lann \tilde P_i, \tilde P_j \rann\right)_{0\le \al,\be \le 3}=
\begin{pmatrix}
 & & \tfrac32  & \\
  &  \tfrac32  &  & \\
\tfrac32  & & & \\
& & &  \tfrac98 y^2
\end{pmatrix}
\quad(=S, \ \text{say}).
\]
Regarding the normalization of the solution, we achieve the analogue
\[
\begin{pmatrix}
 & & \!\frac{d}{m_1}\!  & \\
  & \!\frac{d}{m_1}\!   &  & \\
\!\frac{d}{m_1}\!   & & & \\
& & & \!\frac{1}{m_2}\! 
\end{pmatrix}
\quad=
\begin{pmatrix}
 & & \frac{3}{2}  & \\
  & \frac{3}{2}  &  & \\
\frac{3}{2}   & & & \\
& & & \frac{1}{2} 
\end{pmatrix}
\]
of the Ansatz for $Q_0(I + \h Q_1 + \cdots )$ in section \ref{direct} if we take $y=\tfrac23$.  

This allows us to define an action of $p$  (abstract orbifold quantum multiplication by $p$) on
$A\otimes\C[r^{\pm 1}]$, where $A$ is the vector space whose $\C$-basis 
vectors are denoted by $1,p,p^2,\one_{\frac12}$.  
The matrix of the action with respect to this basis is, by definition,
the matrix $\tilde \om$.   As in section \ref{direct}, we may also introduce a grading 
by defining 
$\age \one_{f_i}=
\tfrac12 \vert \tilde P_{u_1+\cdots+ u_{i-1}}\vert$, and this gives:
\[
\begin{array}{r|c|c|c|}
\hline
\vphantom{\dfrac AA}
\age \one_{0} = 0
& \vert 1\vert=0  &   \vert p\vert=2  &   \vert p^2\vert=4  
\\
\hline
\vphantom{\dfrac AA}
\age \one_{\frac12} = 1
& & &  \vert \one_{\frac12}\vert=2
  \\
 \hline
\end{array}
 \]
 The action is compatible with this grading (i.e.\ the action of $p$ increases degree by $2$).
 
We also have the pairing $(\ ,\ )$ whose matrix is $S$.   The Frobenius condition 
$(p\circ a,b)=(a,p\circ b)$ (for any $a,b\in A$) is satisfied by construction (see the discussion following Definition 6.14 of \cite{Gu08}); in matrix terms this is $S^{-1}\tilde\om^t S=\tilde\om$.   We remark that this holds for any value of $y$, not just $y=\tfrac23$.

The module action reproduces the first two rows of the following table of orbifold quantum products obtained by Corti  (\cite{Co08}):
\renewcommand{\AA}{\tfrac16 q^{\tfrac {A^A}A} \one_{\tfrac2{A_A}}}
\[
\begin{array}{c|ccccc}
\vphantom{\AA}
 & \ \ 1\ \  & \ \ p\ \  &\ \  p^2\ \  & \ \ \one_{\frac12}\ \ 
\\
\hline
\vphantom{\AA}
1  & 1 & p & p^2 & \one_{\frac12}
 \\
\vphantom{\AA}
p  &  & p^2+12r^2+3r \one_{\frac12}  & 12r^2p & r p
\\
\vphantom{\AA}
p^2  &  &  & 108r^4+36r^3\one_{\frac12} & 12r^3
\\
\vphantom{\AA}
\one_{\frac12}  &  &  &  & \tfrac13 p^2 - 3r\one_{\frac12}
\\
\end{array}
 \]
Furthermore, $S$ agrees with the matrix of the orbifold Poincar\'e pairing from \cite{Co08}.

If it is assumed that the module action extends to a commutative associative abstract quantum product operation which satisfies the Frobenius condition 
$(c\circ a,b)=(a,c\circ b)$ (for any $a,b,c\in A$), then it follows from the first two rows of the table that
\begin{align*}
p^2\, \circ p^2\;  &= 108r^4+36r^3\one_{\frac12} +xr^3(r-\one_{\frac12})
\\
p^2\,  \circ \one_{\frac12} &= 12r^3 -\tfrac x3r^2 (r-\one_{\frac12})
\\
\one_{\frac12} \circ \one_{\frac12} &= \tfrac13 p^2 + \tfrac x9 r(r-\one_{\frac12})
\end{align*}
for some real scalar $x$.  It follows that
$(\one_{\frac12} \circ \one_{\frac12}, \one_{\frac12})=-\tfrac32r
+\tfrac x9 r (r-\one_{\frac12})$.
The condition $x=0$ is equivalent to 
\[
(\one_{\frac12} \circ \one_{\frac12}, \one_{\frac12})=-\tfrac32r.
\]
and Corti computed this as a Gromov-Witten invariant.  The ambiguity involving
$r-\one_{\frac12}$ is unavoidable in our construction as the second row of the table  already tells us that 
$p\circ(r-\one_{\frac12})=0$.
However, for any $x$ we do obtain an abstract quantum product operation which satisfies the Frobenius condition.
\qed
\end{example}
 
Returning to the general theory, let us mention an alternative interpretation of our method, which 
explains informally our assumptions (H1)-(H3).  The significance of (H1) is that it is a necessary condition for the natural pairing $\lann\ ,\ \rann$ to be \ll flat\rrr, i.e.\ for condition (H2).  Having such a flat pairing is, in turn, a necessary condition for being able to carry out the Gram-Schmidt orthonormalization procedure, which is what (H3) accomplishes.   From the the Birkhoff factorization point of view, our method utilizes $L_+$ rather than $L_-$, as we have already mentioned; more accurately,  it utilizes the transformation \ll$P_i\mapsto L_+^{-1}\cdot P_i$\rr
in the D-module, which is essentially the Gram-Schmidt process. 

It may appear at first sight that this could be done in many inequivalent ways.  
However (with suitable initial conditions, as in Example \ref{corti}), the final basis
$\tilde P_0, \dots , \tilde P_{s-2}$ is unique, and this may be explained as follows.
Step 2 involves a Birkhoff factorization of the form \ll$L=L_-\ga L_+$\rrr.  The Frobenius property is satisfied if and only if $L_-$ is a twisted loop with respect to the involution defined by $S$, i.e.\ $S^{-1}(L_-^t)^{-1}S=L_-(-\h)$ (section 6.5 of \cite{Gu08}).  Now, if there exists {\em some} twisted $L$, for example, from any Gram-Schmidt orthonormalization, and $\ga$ is twisted, then $L_-$ must also be twisted, as the Birkhoff decomposition is valid also for the twisted loop group.  By the uniqueness of the (normalized)  Birkhoff decomposition, we always obtain the same $L_-$.   Thus, any Gram-Schmidt orthonormalization followed by a Birkhoff factorization produces the same
$\tilde P_0, \dots , \tilde P_{s-2}$.

Thus, the role of the loop $\ga$ is to compensate for the non-flatness of the pairing 
$\lann\ ,\ \rann$.   It does this by modifying the original $D^\h$-module $\calM$ (with basis $P_0,\dots,P_{s-2}$) to a new $D^\h$-module 
with basis $\tilde P_0,\dots,\tilde P_{s-2}$, which is a submodule of
$\calM\otimes_{\C[\h]} \C[\h,\h^{-1}]$.  This phenomenon is related to the failure of the action of $p$ to be cyclic, in the hypersurface case.  We thank Hiroshi Iritani for emphasizing to us the significance of this, cf.\ \cite{IrXX}.

{\em
\noindent
Department of Mathematics\newline
Faculty of Science and Engineering\newline
Waseda University\newline
3-4-1 Okubo, Shinjuku, Tokyo 169-8555\newline
JAPAN
   
\noindent
E-mail: martin@waseda.jp

\noindent
Mathematisches Institut
\newline
WWU M\"unster
\newline
Einsteinstrasse 62
\newline
48149 M\"unster
\newline
GERMANY

\noindent
E-mail: sakai@blueskyproject.net
}

\end{document}